\documentclass{amsart}

\usepackage[mathscr]{eucal}
\usepackage{amssymb}
\usepackage[usenames,dvipsnames]{color}
\usepackage[normalem]{ulem}
\usepackage{amsthm}
\usepackage{bbold}
\usepackage{enumerate}
\usepackage{array}
\usepackage{amsmath}
\usepackage{mathtools}

\usepackage{hyperref}

\hypersetup{colorlinks=true,citecolor=BrickRed,urlcolor=BrickRed,linkcolor=BrickRed}

\numberwithin{equation}{section}
\usepackage[all]{xy}

\newdir{ >}{{}*!/-10pt/\dir{>}}

\SelectTips{cm}{} 

\swapnumbers 

\theoremstyle{definition}

\newtheorem*{conv*}{Conventions}

\newtheorem*{Related work}{Related work}

\theoremstyle{theorem}

\newtheorem{thm}[equation]{Theorem}
\newtheorem{lemma}[equation]{Lemma}
\newtheorem{thm-defi}[equation]{Theorem-Definition}
\newtheorem{prop}[equation]{Proposition}
\newtheorem{cor}[equation]{Corollary}

\newtheorem*{thm*}{Theorem}
\newtheorem*{lemma*}{Lemma}
\newtheorem*{cor*}{Corollary}
\newtheorem*{conj*}{Conjecture}
\newtheorem*{question*}{Question}

\theoremstyle{remark}

\newtheorem{remark}[equation]{Remark}
\newtheorem*{remark*}{Remark}

\newtheorem{example}[equation]{Example}

\newcommand{\nc}{\newcommand}
\nc{\dmo}{\DeclareMathOperator}

\nc{\Ivo}[1]{{\color{OliveGreen}#1}}
\nc{\Don}[1]{{\color{MidnightBlue}#1}}
\nc{\X}[1]{{\color{Violet}#1}}
\nc{\Xout}[1]{\X{\sout{#1}}}
\nc{\Iout}[1]{\Don{\sout{#1}}}
\nc{\Dout}[1]{\Ivo{\sout{#1}}}

\newcommand{\Spec}{\mathop{\mathrm{Spec}}}
\newcommand{\suppR}{\mathop{\mathrm{supp}_R}} 
\newcommand{\Supp}{\mathop{\mathrm{Supp}}} 
\newcommand{\SuppR}{\mathop{\mathrm{Supp}_R}} 

\newcommand{\Spc}{\mathop{\mathrm{Spc}}}
\newcommand{\supp}{\mathop{\mathrm{supp}}}

\newcommand{\cone}{\mathop{\mathrm{cone}}}
\newcommand{\Thick}{\mathop{\mathrm{Thick}}} 
\newcommand{\Loc}{\mathop{\mathrm{Loc}}} 

\newcommand{\unit}{\mathbf{1}} 

\newcommand{\Hom}{\mathrm{Hom}}
\newcommand{\End}{\mathrm{End}}

\newcommand{\Coloc}[1]{\mathit{\Gamma}_{#1}} 

\newcommand{\koszul}{/\!\!/} 
\newcommand{\residue}[1]{K(\mathfrak{#1})} 

\newcommand{\s}{\mathrm{s}}

\begin{document}


\title{Affine weakly regular tensor triangulated categories}
\author{Ivo Dell'Ambrogio}
\author{Donald Stanley}
\date{\today}

\address{Ivo Dell'Ambrogio, Laboratoire de Math\'ematiques Paul Painlev\'e, Universit\'e de Lille~1, Cit\'e Scientifique -- B\^at.~M2, 59665 Villeneuve-d'Ascq Cedex, France.}
\email{ivo.dellambrogio@math.univ-lille1.fr}
\urladdr{http://math.univ-lille1.fr/$\sim$dellambr}

\address{Donald Stanley, Department of Mathematics and Statistics,
University of Regina,
3737 Wascana Parkway,
Regina, Saskatchewan,
S4S 0A2  Canada.}
\email{stanley@math.uregina.ca}

\begin{abstract}
We prove that the Balmer spectrum of a tensor triangulated category is homeomorphic to the Zariski spectrum of its graded central ring, provided the triangulated category is generated by its tensor unit and the graded central ring is noetherian and regular in a weak sense. 
There follows a classification of all thick subcategories, and the result extends to the compactly generated setting to yield a classification of all localizing subcategories as well as the analog of the telescope conjecture. 
This generalizes results of Shamir for commutative ring spectra. 
\end{abstract}


\thanks{}
\thanks{First-named author partially supported by the Labex CEMPI (ANR-11-LABX-0007-01)}
\thanks{}

\maketitle



\section{Introduction and results}
\label{sec:intro}

Let $\mathcal K$ be an essentially small tensor triangulated category, with symmetric exact tensor product $\otimes$ and tensor unit object~$\unit$.
Balmer~\cite{Balmer05a} defined a topological space, the \emph{spectrum} $\Spc \mathcal K$, that allows for the development of a geometric theory of~$\mathcal K$, similarly to how the Zariski spectrum captures the intrinsic geometry of commutative rings; see the survey~\cite{BalmerICM}. 
Among other uses, Balmer's spectrum encodes the classification of the thick tensor ideals of $\mathcal K$ in terms of certain subsets. 
It is therefore of interest to find an explicit description of the spectrum in the examples, but this is usually a difficult problem requiring some in-depth knowledge of each example at hand.

The goal of this note is to show that in some cases a concrete description of the spectrum can be obtained easily and completely formally. 
Let us denote by
\[
R:= \End_\mathcal K^*(\unit) = \bigoplus_{i\in \mathbb Z}\Hom_\mathcal K(\unit, \Sigma^i \unit)
\]
the graded endomorphism ring of the unit, where  $\Sigma \colon \mathcal K\to \mathcal K$ is the suspension functor. In the terminology of~\cite{Balmer10b}, this is the \emph{graded central ring} of~$\mathcal K$.
It is a graded commutative ring and therefore we can consider its spectrum of homogeneous prime ideals, $\Spec R$, equipped with the Zariski topology. As established in~\cite{Balmer10b}, there is always a canonical continuous map 
\[
\rho \colon \Spc \mathcal K \longrightarrow \Spec R
\]
comparing the two spectra. Under some mild hypotheses, e.g.\ when $R$ is noetherian,  $\rho$ can be shown to be surjective, but it is less frequently injective and, when it is, the proof of injectivity is typically much harder. 

Here is our main result:

\begin{thm} \label{thm:main_spectrum}
Assume that $\mathcal K$ satisfies the two following conditions:
\begin{enumerate}[\indent\rm(a)]
\item \label{it:affine} 
$\mathcal K$ is classically generated by $\unit$, i.e., as a thick subcategory:
$\Thick(\unit)= \mathcal K$.
\item \label{it:regular} $R$ is a (graded) noetherian ring concentrated in even degrees and, for every homogeneous prime ideal $\mathfrak p$ of~$R$, the maximal ideal of the local ring $R_{\mathfrak p}$ is generated by a (finite) regular sequence of homogeneous non-zero-divisors.
\end{enumerate}
Then the comparison map $\rho\colon \Spc \mathcal K \stackrel{\sim}{\to} \Spec R$ is a homeomorphism.
\end{thm}

As in the title, we may refer to a tensor triangulated category $\mathcal K$ satisfying hypotheses~\eqref{it:affine} and \eqref{it:regular} as being \emph{affine} and \emph{weakly regular}, respectively.
Note that $R$ being noetherian implies that $R^0=\End_\mathcal K(\unit)$ is a noetherian ring and that $R$ a finitely generated $R^0$-algebra, by~\cite{GotoYamagishi83}.

The next result is an easy consequence of the theorem. Here $\Supp_R H^*X$ denotes the (big) Zariski support of the cohomology graded $R$-module $H^*X:= \Hom^*_\mathcal K(\unit, X)$.

\begin{cor} \label{cor:main_thick}
If $\mathcal K$ and $R$ are as in the theorem, then there exists a canonical inclusion-preserving bijection 
\[
\xymatrix{
\big\{ \textrm{thick subcategories } \mathcal C \textrm{ of } \mathcal K  \big\} \ar@<3pt>[r]^-\sim & 
\ar@<2pt>[l] \big\{ \textrm{specialization closed subsets } V  \textrm{ of } \Spec R \big\}
}
\]
mapping a thick subcategory~$\mathcal C$ to $V=\bigcup_{X\in \mathcal C} \Supp_R H^*X$ and a specialization closed subset~$V$ to $\mathcal C=\{X\in \mathcal K \mid \Supp_R H^*X \subseteq V\}$. 
\end{cor}

In many natural examples, $\mathcal K$ occurs as the subcategory $\mathcal T^c$ of compact objects in a compactly generated tensor triangulated category~$\mathcal T$. By the latter we mean a compactly generated triangulated category $\mathcal T$ equipped with a symmetric monoidal structure $\otimes$ which preserves coproducts and exact triangles in both variables, and such that the compact objects form a tensor subcategory~$\mathcal T^c$ 
(that is: $\unit $ is compact and the tensor product of two compact objects is again compact).

In this case, the same hypotheses allow us to classify also the localizing subcategories of~$\mathcal T$, thanks to the stratification theory of compactly generated categories due to Benson, Iyengar and Krause \cite{BensonIyengarKrause11}. 
The \emph{support} $\suppR X \subseteq \Spec R$ of an object $X\in \mathcal T$ is defined in \cite{BensonIyengarKrause08}, and can be described as the set 
\[
\suppR X = \{ \mathfrak p \in \Spec R \mid X \otimes \residue{p} \neq 0 \}
\]
where the \emph{residue field object} $\residue{p}$ of a prime ideal $\mathfrak p$ is an object whose cohomology is the graded residue field of $R$ at~$\mathfrak p$; see~\S\ref{sec:thick}.

\begin{thm} \label{thm:main_localizing}
Let $\mathcal T$ be a compactly generated tensor triangulated category with compact objects $\mathcal K:=\mathcal T^c$ and graded central ring~$R$ satisfying conditions \eqref{it:affine} and~\eqref{it:regular}.
Then we have the following canonical inclusion-preserving bijection:
\[
\xymatrix{
\big\{ \textrm{localizing subcategories } \mathcal L \subseteq \mathcal T  \big\} \ar@<3pt>[r]^-\sim & 
\ar@<2pt>[l] \big\{ \textrm{subsets } S\subseteq \Spec R \big\} \,.
}
\]
The correspondence sends a localizing subcategory~$\mathcal L$ to $S=\bigcup_{X\in \mathcal L} \suppR X$, and an arbitrary subset~$S$ to $\mathcal L=\{X\in \mathcal T \mid \suppR X \subseteq S\}$. 
Moreover, the bijection restricts to localizing subcategories $\mathcal L=\Loc(\mathcal L\cap \mathcal K)$ which are generated by compact objects on the left and to specialization closed subsets $S=\bigcup_{\mathfrak p\in S} V(\mathfrak p)$ on the right.
\end{thm}

Note that here the affine condition \eqref{it:affine} is equivalent to requiring that $\mathcal T$ is generated by $\unit$ as a localizing subcategory. 
As $\Supp_R H^* X = \suppR X$ for all compact objects $X\in \mathcal K$, one sees easily that in the compactly generated case Theorem~\ref{thm:main_spectrum} and Corollary~\ref{cor:main_thick} are also a consequence of Theorem~\ref{thm:main_localizing}. 
 
The next corollary is another by-product of stratification. Recall that a localizing subcategory $\mathcal L\subseteq \mathcal T$ is \emph{smashing} if the inclusion functor $\mathcal L\hookrightarrow \mathcal T$ admits a coproduct-preserving right adjoint.

\begin{cor}[The telescope conjecture in the affine weakly regular case] \label{cor:telescope}
In the situation of Theorem~\ref{thm:main_localizing}, every smashing subcategory of $\mathcal T$ is generated by a set of compact objects of~$\mathcal T$.
\end{cor}

A few special cases of our formal results had already been observed, such as when $R$ is even periodic and of global dimension at most one; see \cite{DellAmbrogioTabuada12}.
We now consider some more concrete examples.

\begin{center} $***$ \end{center}

\begin{example} \label{ex:cdga}
Let $A$ be a commutative dg~algebra and $D(A)$ its derived category of dg~modules.
Then $D(A)$ is an affine compactly generated tensor triangulated category with respect to the standard tensor product $\otimes=\otimes^\mathrm{L}_A$, and $R= H^{*}A$ is the cohomology algebra of~$A$; thus if the latter satisfies~\eqref{it:regular} all our results apply to~$D(A)$. Actually, in this example we can improve our results a little by eliminating the hypothesis that $R$ is even and that the elements of the regular sequences are non-zero-divisors:
\end{example}

\begin{thm} \label{thm:cdga}
Let $A$ be a commutative dg~algebra such that its graded cohomology ring $R=H^*A$ is noetherian and such that every local prime $\mathfrak p R_{\mathfrak p}$ is generated by a finite regular sequence. Then all the conclusions of Theorems~\ref{thm:main_spectrum} and~\ref{thm:main_localizing} and of Corollaries~\ref{cor:main_thick} and~\ref{cor:telescope} hold for $\mathcal T=D(A)$ and $\mathcal K=D(A)^c$.
\end{thm}

We can apply this for instance to a graded polynomial algebra with any choice of grading for the variables, seen as a strictly commutative formal dg algebra.

\begin{example}
Let $A$ be a commutative $S$-algebra (a.k.a.\ commutative highly structured ring spectrum), and let $D(A)$ be its derived category. 
(This covers Example~\ref{ex:cdga}, as commutative dga's can be seen as commutative $S$-algebras.)
Then $D(A)$ is an affine compactly generated tensor triangulated category, and $R= \pi_{*} A$ is the stable homotopy algebra of~$A$; thus if the latter satisfies~\eqref{it:regular} all our results apply to~$D(A)$.
Shamir~\cite{Shamir12} already treated this example under the additional hypotheses that $\pi_*A$  has finite Krull dimension. 
Working with $\infty$-categories and $\mathrm E_\infty$-rings, Mathew~\cite[Theorem~1.4]{Mathew14app} established the classification of thick subcategories as in Corollary~\ref{cor:main_thick} for the case when $\pi_*A$ is even periodic and $\pi_0A$ regular noetherian. Remarkably, in the special case of $S$-algebras defined over~$\mathbb Q$, Mathew was also able to prove the classification of thick subcategories only assuming $\pi_*A$ noetherian, i.e., without any regularity hypothesis; see~\cite[Theorem~1.4]{Mathew14}.
\end{example}

The next two well-known examples show that neither hypothesis \eqref{it:affine} nor \eqref{it:regular} can be weakened with impunity.

\begin{example}
The derived category $\mathcal T=D(\mathbb P_k^1)$ of the projective line over a field~$k$ is an example
where $R=\End^*(\unit)\simeq k$ certainly satisfies \eqref{it:regular} but \eqref{it:affine} does not hold. Indeed $\rho$ can be identified with the structure map $\mathbb P_k^1 \to \Spec k$ and is therefore far from injective in this case; see \cite[Remark~8.2]{Balmer10b}.
\end{example}

\begin{example}
If $\mathcal T=D(A)$ is the derived category of a commutative (ungraded) ring~$A$, Theorem~\ref{thm:main_spectrum} and the classification of thick subcategories always hold by a result of Thomason~\cite{Thomason97} (see \cite[Proposition~8.1]{Balmer10b});  the classification of localizing subcategories and the telescope conjecture hold if $A$ is noetherian by Neeman~\cite{Neeman92a}.
On the other hand, Keller~\cite{Keller94b} found examples of \emph{non-noetherian} rings~$A$ for which the two latter results fail.
\end{example}

In view of these examples, it would be interesting to know how far our weak regularity hypothesis~\eqref{it:regular} can be weakened in general. Would noetherian suffice?

\section{Recollections}
\label{sec:recollections}

Let $\mathcal K$ be an essentially small tensor triangulated category.

For any two objects $X,Y\in \mathcal K$, consider the $\mathbb Z$-graded group $\Hom^*_\mathcal K(X,Y)=\bigoplus_{i\in \mathbb Z} \Hom_{\mathcal K}(X, \Sigma^i Y)$.
Recall that the symmetric tensor product of $\mathcal K$ canonically induces on $R:= \Hom^*_\mathcal K(\unit,\unit)$ the structure of a graded commutative\footnote{To be precise, graded commutativity means here that $fg = \epsilon^{|f| |g|} gf$ for any  two homogeneous elements $f\in \Hom_\mathcal K(\unit, \Sigma^{|f|}\unit)$ and $g\in \Hom_\mathcal K(\unit, \Sigma^{|g|}\unit)$, where $\epsilon \in R^0$ is a constant with $\epsilon^2=1$ induced by the symmetry isomorphism $\Sigma \unit \otimes \Sigma \unit \stackrel{\sim}{\to} \Sigma \unit \otimes \Sigma \unit$. In most cases we have $\epsilon = -1$, e.g.\ if $\mathcal K$ admits a symmetric monoidal model, but usually no extra difficulty arises by allowing the general case. Of course, this is immaterial for $R$ even.} ring, and on each $\Hom^*_\mathcal K(X,Y)$ the structure of a (left and right) graded $R$-module. The composition of maps in $\mathcal K$ and the tensor functor $-\otimes-$ are (graded) bilinear for this action. 
See \cite[\S3]{Balmer10b} for details.

Since we are using cohomological gradings, we write $H^*X$ for the $R$-module $\Hom^*_\mathcal K(\unit, X)$ and call it the \emph{cohomology of~$X$}.
\subsection*{Supports for graded modules}
We denote by $\Spec R$ the Zariski spectrum of all homogeneous prime ideals in~$R$. 
If $M$ is an $R$-module (always understood to be graded) and $\mathfrak p\in \Spec R$, the graded localization of $M$ at $\mathfrak p$ is the $R$-module $M_\mathfrak p$ obtained by inverting the action of all the \emph{homogeneous} elements in $R\smallsetminus \mathfrak p$. The \emph{big support} of $M$ is the following subset of the spectrum: 
\[
\SuppR M = \{ \mathfrak p \in \Spec R \mid M_\mathfrak p \neq 0\} \,.
\]
Since our graded ring $R$ is noetherian we also dispose of the \emph{small support}, defined in terms of the indecomposable injective $R$-modules $E(R/\mathfrak p)$:
\[
\suppR M = \{\mathfrak p \mid E(R/\mathfrak p) \textrm{ occurs in the minimal injective resolution of }M \} \,.
\]
We recall from \cite[\S2]{BensonIyengarKrause08} some well-known properties of supports. In general we have $\suppR M \subseteq \SuppR M$. If $M$ is finitely generated, these two sets are equal and also coincide with the Zariski closed set $V(\mathrm{Ann}_R M)$. For a general~$M$, $\SuppR M$ is always \emph{specialization closed}: if it contains any point $\mathfrak p$ then it must contain its closure $V(\mathfrak p)=\{\mathfrak q\mid \mathfrak p\subseteq \mathfrak q\}$. In fact $\SuppR M$ is equal to the specialization closure of $\suppR M$: $\SuppR M= \bigcup_{\mathfrak p\in \suppR M} V(\mathfrak p)$.
The small support plays a fundamental role in the Benson-Iyengar-Krause stratification theory, but in this note it will only appear implicitly.

The next lemma follows by a standard induction on the length of the objects.

\begin{lemma} \label{lemma:supports} If $\mathcal K= \Thick(\unit)$ is affine and $R$ is noetherian, the graded $R$-module $\Hom^*_\mathcal K(X,Y)$ is finitely generated for all $X,Y\in \mathcal K$.
\qed
\end{lemma}

\subsection*{The comparison map of spectra}
Recall from \cite{Balmer05a} that, as a set, the spectrum $\Spc \mathcal K$ is defined to be the collection of all proper thick subcategories $\mathcal P\subsetneq \mathcal K$ which are \emph{prime tensor ideals}: $X\otimes Y\in \mathcal P$ $\Leftrightarrow$ $X\in \mathcal P$ or $Y\in \mathcal P$.
For every $\mathcal P\in \Spc \mathcal K$, let $\rho_\mathcal K(\mathcal P)$ denote the ideal of $\Spec R$ generated by the set of homogeneous elements $\{f\colon \unit \to \Sigma^{|f|}\unit \mid \cone(f) \not\in \mathcal P \}$.
By \cite[Theorem~5.3]{Balmer10b}, the assignment $\mathcal P\mapsto \rho_{\mathcal K}(\mathcal P)$ defines a continuous map $\rho_\mathcal K\colon \Spc \mathcal K\to \Spec R$, natural in~$\mathcal K$. 
Moreover, the two spaces $\Spc \mathcal K$ and $\Spec R$ are \emph{spectral} in the sense of Hochster~\cite{Hochster69}, and $\rho_\mathcal K$ is a \emph{spectral map} in that the preimage of a compact open set is again compact. 

\begin{lemma} \label{lemma:bij_homeo}
If $\rho_\mathcal K$ is bijective then it is a homeomorphism.
\end{lemma}
\begin{proof}
This is an immediate consequence of \cite[Proposition~15]{Hochster69}, which says that for a spectral map of spectral topological spaces to be a homeomorphism it suffices that it is an order isomorphism for the specialization order of the two spaces.
Recall that the specialization order is defined for the points of any topological space by $x \geq y \Leftrightarrow x \in \overline{\{y\}}$.
Indeed $\rho:=\rho_\mathcal K$ is inclusion reversing, $\mathcal Q \subseteq \mathcal P \Leftrightarrow \rho(\mathcal Q) \supseteq \rho(\mathcal P)$, hence it maps the closure $\overline{\{\mathcal P\}}= \{\mathcal Q\mid \mathcal Q\subseteq \mathcal P\}$ in $\Spc \mathcal K$ of any point $\mathcal P$ to the Zariski closure $V(\rho(\mathcal P))=\{\mathfrak q\mid \mathfrak q\supseteq \rho(\mathcal P)\}$ in $\Spec R$ of the corresponding point.
\end{proof}

\subsection*{Central localization}
For every prime ideal $\mathfrak p$ of the graded central ring $R$ of $\mathcal K$, there exists by \cite[Theorem~3.6]{Balmer10b} a tensor triangulated category $\mathcal K_\mathfrak p$ having the same objects as $\mathcal K$ and such that its graded Hom modules are the localizations
\[
\Hom^*_{\mathcal K_\mathfrak p}(X,Y) = \Hom^*_{\mathcal K}(X,Y)_{\mathfrak p} \,.
\]
In particular the graded central ring of $\mathcal K_\mathfrak p$ is the local ring~$R_\mathfrak p$. There is a canonical exact functor $q_\mathfrak p \colon \mathcal K\to \mathcal K_\mathfrak p$, which is in fact the Verdier quotient by the thick tensor ideal generated by $\{\cone(f)\in \mathcal K \mid f \in R\smallsetminus \mathfrak p \textrm{ homogeneous} \}$. For emphasis, we will sometimes write $X_\mathfrak p$ for $X=q_{\mathfrak p}X$ when considered as an object of~$\mathcal K_\mathfrak p$.

Clearly if $\mathcal K$ is generated by $\unit$ then $\mathcal K_\mathfrak p$ is generated by $\unit_\mathfrak p$. 
Later we will use the fact that if a tensor triangulated category is generated by its unit then every thick subcategory is automatically a tensor ideal.

Let $\ell_{\mathfrak p}\colon R\to R_\mathfrak p$ denote the localization map between the graded central rings of the two categories. By \cite[Theorem~5.4]{Balmer10b}, we have a pullback square of spaces 
\begin{equation} \label{eq:pullback_spectra}
\xymatrix{
\Spc (\mathcal K_\mathfrak p) \, \ar[d]_{\rho_{\mathcal K_\mathfrak p}} \ar@{^{(}->}[rr]^-{\Spc (q_{\mathfrak p})} &&
 \Spc(\mathcal K) \ar[d]^{\rho_{\mathcal K}} \\
\Spec(R_{\mathfrak p}) \, \ar@{^{(}->}[rr]^-{\Spec(\ell_{\mathfrak p})} && \Spec(R)
}
\end{equation}
where the horizontal maps are injective. 

\subsection*{Koszul objects}
We adapt some convenient notation  from \cite{BensonIyengarKrause08}. For any object $X\in \mathcal K$ and homogeneous element $f\in R$, let $X\koszul f := \cone(f \cdot X)$ be any choice of mapping cone for the map $f\cdot X\colon \Sigma^{-|f|}X\to X$ given by the $R$-action.
If $f_1,\ldots, f_n$ is a finite sequence of homogeneous elements, define recursively $X_0:=X$ and  $X_i:=X\koszul (f_1,\ldots,f_i):= (X\koszul (f_1,\ldots, f_{i-1}) )\koszul f_i$ for $i\in \{1,\ldots, n\}$. Thus by construction we have exact triangles
 \begin{equation}  \label{eq:triangle_Xi}
\xymatrix{
\Sigma^{-|f_i|}  X_{i-1} \ar[r]^-{f_i \cdot X_{i-1}} & X_{i-1} \ar[r] & X_i \ar[r] & \Sigma^{-|f_i|+1} X_{i-1} \,,
}
\end{equation}
and moreover, since the tensor product is exact, we have isomorphisms
\[
X\koszul (f_1,\ldots,f_i) 
\simeq X \otimes \unit \koszul f_1 \otimes \ldots \otimes \unit\koszul f_i
\] 
for all~$i\in \{1,\ldots, n\}$. In the following, we will perform this construction \emph{inside the $\mathfrak p$-local category~$\mathcal K_\mathfrak p$}.

We need the following triangular version of the Nakayama lemma, for $\mathcal K$ affine. 

\begin{lemma} \label{lemma:ttNAK}
If $X\in \mathcal K_\mathfrak p$ is any object and $f_1,\ldots,f_n$ is a set of homogenous generators for $\mathfrak pR_\mathfrak p$, then in $\mathcal K_\mathfrak p$ we have $X=0$ if and only if $X\koszul (f_1,\ldots, f_n)=0$.
\end{lemma}

\begin{proof} 
Since $\mathcal K$ and thus $\mathcal K_\mathfrak p$  are generated by their tensor unit, is suffices to show that $H^*X_\mathfrak p =0$ if and only if $H^*(X\koszul (f_1,\ldots, f_n))_\mathfrak p=0$, and the latter can be proved as in \cite[Lemma 5.11\,(3)]{BensonIyengarKrause08}. 
We give the easy argument for completeness.

With the above notation, by taking cohomology $H^*=\Hom^*_{\mathcal K_{\mathfrak p}}(\unit_\mathfrak p, -)$ of the triangle \eqref{eq:triangle_Xi} of $\mathcal K_\mathfrak p$ we obtain the long exact sequence of $R_\mathfrak p$-modules
 \begin{equation*}
\xymatrix{
\ldots \ar[r] & H^{*-|f_i|} X_{i-1} \ar[r]^-{f_i} & H^{*} X_{i-1} \ar[r] & H^* X_i \ar[r] & H^{*-|f_i|+1} X_{i-1} 
\ar[r] & \ldots
}
\end{equation*}
where each module is finitely generated by Lemma~\ref{lemma:supports}.
Since $f_i \in \mathfrak p$, if $H^* X_{i-1}\neq 0$ the first map in the sequence is not invertible by the Nakayama lemma, hence $H^*X_i\neq0$. The evident recursion shows that $H^*X\neq 0$ implies $H^*X_n\neq0$. 
The very same exact sequences also show that if $H^*X=0$ then $H^*X_n=0$.
\end{proof}

\section{Thick subcategories}
\label{sec:thick}

Assume from now on that $\mathcal K$ satisfies the conditions \eqref{it:affine} and \eqref{it:regular} of Theorem~\ref{thm:main_spectrum}. 
\subsection*{Residue field objects}
By hypothesis, for every prime ideal $\mathfrak p\in \Spec R$ there exists a regular sequence $f_1,\ldots,f_n$ of homogeneous non-zero-divisors of $R_{\mathfrak p}$ which generate the ideal~$\mathfrak p R_\mathfrak p$.
Choose one such sequence once and for all, and construct the associated Koszul object 
\[
K(\mathfrak p) := \unit_\mathfrak p \koszul (f_1,\ldots, f_n) \simeq \unit_\mathfrak p\koszul f_1 \otimes \ldots \otimes \unit_\mathfrak p\koszul f_n
\]
in the $\mathfrak p$-local tensor triangulated category $\mathcal K_\mathfrak p$. 

\begin{lemma} \label{lemma:zero_action}
For every object $X\in \mathcal K_\mathfrak p$ and every~$i\in \{1,\ldots,n\}$, each element $f$ of the ideal $(f_1,\ldots,f_i)\subset R_\mathfrak p$ acts as zero on $X \koszul (f_1,\ldots,f_i)$, i.e., $f\cdot {X \koszul (f_1,\ldots,f_i)}=0$.
\end{lemma}

\begin{proof}
Recall that, as an immediate consequence of the $R_\mathfrak p$-bilinearity of the composition in~$\mathcal K_\mathfrak p$, the elements of $R_\mathfrak p$ acting as zero on an object $Y$ form an ideal (coinciding with the annihilator of the $R_\mathfrak p$-module $\Hom_{\mathcal K_\mathfrak p}^*(Y,Y)$). Thanks to the isomorphism $X\koszul (f_1,\ldots,f_i) \simeq X\otimes \unit_\mathfrak p\koszul f_1 \otimes \ldots \otimes \unit_\mathfrak p\koszul f_i$ and the $R_\mathfrak p$-linearity of the tensor product, it will therefore suffice to prove that $f_i$ acts as zero on $\unit_\mathfrak p\koszul f_i$. Consider the following commutative diagram
\[
\xymatrix{
\Sigma^{-|f_i|}\unit_\mathfrak p \ar[r]^-{f_i} &
 \unit_\mathfrak p \ar[r]^-{g} \ar[d]_{f_i} \ar[dr]_0 & 
  \unit_\mathfrak p \koszul f_i \ar[d]^{f_i} \ar[r] & \Sigma^{-|f_i| +1} \unit_\mathfrak p \ar@{..>}[dl]^-{h} \\
& \Sigma^{|f_i|} \unit_\mathfrak p \ar[r]_-{g} & \Sigma^{|f_i|}\unit_\mathfrak p \koszul f_i &
}
\]
where the top row is the exact triangle defining $\unit_\mathfrak p\koszul f_i$. 
Being the composite of two consecutive maps in a triangle, $g f_i $ is zero. Up to a suspension, this is also the diagonal map in the square. Hence $f_i \cdot \unit_\mathfrak p \koszul f_i$ factors through a map $h$ as pictured.
Since $R$ is even by hypothesis, then $R_\mathfrak p$ is even and we claim that also
\begin{equation} \label{eq:even_koszul}
H^n(\unit_\mathfrak p\koszul f_i) = 0 \quad \textrm{ for all odd } n\,.
\end{equation}
This implies $h=0$ and therefore $f_i\cdot \unit_\mathfrak p\koszul f_i=0$, as required.
To prove the claim, note that the defining triangle of $\unit_\mathfrak p\koszul f_i$ induces the exact sequence
\[
\xymatrix{
R_\mathfrak p^{n-|f_i|} \ar[r]^-{f_i}  & R_\mathfrak p^n \ar[r] & H^n(\unit_\mathfrak p\koszul f_i) \ar[r] &
R_\mathfrak p^{n-|f_i|+1} \ar[r]^-{f_i} & R_\mathfrak p^{n+1}
}
\]
where the first and last maps are injective by the hypothesis that $f_i$ is a non-zero-divisor in~$R_\mathfrak p$. Thus \eqref{eq:even_koszul}, and even $H^*(\unit_\mathfrak p\koszul f_i)\simeq R_\mathfrak p/ (f_i)$, follow immediately.
\end{proof}

\begin{cor} \label{cor:vector_space}
$H^*(X \otimes \residue{p})$ is a graded $k(\mathfrak p)$-vector space for every $X\in \mathcal K_\mathfrak p$.
\end{cor}

\begin{proof}
By Lemma~\ref{lemma:zero_action} together with the $R$-linearity of the tensor product, each $f\in \mathfrak p R_\mathfrak p$ acts as zero on $X \otimes \residue{p}\simeq X \otimes \unit_\mathfrak p\koszul (f_1,\ldots, f_n)$. 
Therefore all such $f$ also act as zero on $H^*(X \otimes \residue{p})$ by the $R$-linearity of composition.
\end{proof}

\begin{lemma}  \label{lemma:koszul_quotient}
There is an isomorphism $H^*(\unit_\mathfrak p\koszul(f_1,\ldots,f_i) )\simeq R_\mathfrak p/(f_1,\ldots,f_i)$ of $R$-modules for all~$i\in\{1,\ldots,n\}$. 
In particular $H^*K(\mathfrak p) $ is isomorphic to the residue field $k(\mathfrak p):= R_{\mathfrak p}/\mathfrak p R_\mathfrak p$. 
\end{lemma}

\begin{proof}
Write $C_0=\unit_\mathfrak p$ and $C_i:= \unit_\mathfrak p\koszul (f_1,\ldots,f_{i})$ for short.
Then $K(\mathfrak p)= C_n$, and for all $i\in \{1,\ldots, n\}$ we have exact triangles
 \begin{equation*} 
\xymatrix{
\Sigma^{-|f_i|} C_{i-1} \ar[r]^-{f_i \cdot C_{i-1}} &  C_{i-1} \ar[r] & C_i \ar[r] & \Sigma^{-|f_i|+1} C_{i-1} \,.}
\end{equation*}
The claim follows by recursion on~$i$. Indeed $H^* C_0 = R_\mathfrak p$, and assume that $H^*C_{i-1}\simeq R_\mathfrak p/(f_1,\ldots,f_{i-1})$. Then the above triangle induces an exact sequence 
\[
\xymatrix{
 H^{*-|f_i|}C_{i-1} \ar[r]^-{f_i} &
  H^{*}C_{i-1} \ar[r] &
   H^*C_i \ar[r] &
    H^{*-|f_i|+1} C_{i-1} \ar[r]^-{f_i} &
     H^{*+1} C_{i-1} }
\]
where the first and last maps are injective because by hypothesis $f_i$ is a nonzero divisor in the ring $R_\mathfrak p/(f_1,\ldots,f_{i-1})$.
We thus obtain a short exact sequence 
$0\to f_i  R_\mathfrak p/(f_1,\ldots, f_{i-1})\to R_\mathfrak p/(f_1,\ldots, f_{i-1}) \to H^*C_i\to 0$, 
proving the claim for~$i$.
\end{proof}

\begin{remark} \label{rem:hyp}
Of the weak regularity hypothesis~\eqref{it:regular}, the proof of Lemma~\ref{lemma:koszul_quotient} only uses that $f_1,\ldots,f_n$ is a regular sequence, while the proof of Lemma~\ref{lemma:zero_action} only uses that the $f_i$ are non-zero-divisors in~$R_\mathfrak p$ and that the  ring~$R$ is even. These are the only places where we make use of these assumptions (the noetherian hypothesis, on the other hand, will be needed on several occasions).
Note that, although we already know  by Corollary~\ref{cor:vector_space} that $H^*\residue{p}$ is a $k(\mathfrak p)$-vector space, for the next proposition we also need it to be one-dimensional as per Lemma~\ref{lemma:koszul_quotient}. 
\end{remark}

\begin{prop} \label{prop:decomposition}
For all $\mathfrak p\in \Spec R$ and $X\in \mathcal K_\mathfrak p$, the tensor product $X\otimes K(\mathfrak p)$ decomposes into a coproduct of shifted copies of the residue field object: 
\[\coprod_{\alpha} \Sigma^{n_\alpha} \residue{p} \stackrel{\sim}{\longrightarrow} X\otimes \residue{p}\,.\]
\end{prop}

\begin{proof}
By Corollary~\ref{cor:vector_space} we know that $H^*(X\otimes \residue{p})$ is a graded $k(\mathfrak p)$-vector space. 
Choose a graded basis  $\{x_\alpha\}_\alpha$, corresponding to a morphism 
$\coprod_\alpha \Sigma^{n_{\alpha}} \unit_\mathfrak p \to X\otimes \residue{p}$.
We will show that this map extends nontrivially to the Koszul object
\[
(\coprod_\alpha \Sigma^{n_\alpha} \unit_\mathfrak p)\koszul (f_1,\ldots,f_n) = \coprod_\alpha (\Sigma^{n_\alpha} \unit_\mathfrak p \koszul (f_1,\ldots,f_n)) \,.
\] 
For this, it will suffice to extend each individual map $x_\alpha\colon \Sigma^{n_\alpha}\unit_\mathfrak p \to X\otimes \residue{p}$.
As before, we proceed recursively along the regular sequence $f_1,\ldots,f_n$. 
Consider the following commutative diagram
\[
\xymatrix{
\Sigma^{n_\alpha - |f_1|} \unit_\mathfrak p \ar[d]_{\Sigma^{-|f_1|}x_\alpha} \ar[r]^-{f_1} \ar[dr]_0 &
 \Sigma^{n_\alpha} \unit_\mathfrak p \ar[r] \ar[d]^{x_\alpha} &
  \Sigma^{n_\alpha}\unit_\mathfrak p \koszul f_1 \ar[r] \ar@{..>}[dl]^{x_\alpha^1} & \\
\Sigma^{-|f_1|} X\otimes \residue{p} \ar[r]_-{f_1 =0 } & X\otimes \residue{p} &&
}
\]
where the top row is a rotation of the defining triangle for $\unit_\mathfrak p\koszul f_1$. The left-bottom composite vanishes because $f_1$ acts trivially on $X\otimes \residue{p}$ by Lemma~\ref{lemma:zero_action}. 
Hence we obtain the map $x_\alpha^1$ on the right.
Note that $x_\alpha^1\neq 0$ because $x_\alpha\neq0$.
Now we repeat the procedure for $i=2,\ldots,n$, using the triangle 
\[
\xymatrix{
\Sigma^{-|f_i|}  \unit_\mathfrak p \koszul (f_1,\ldots,f_{i-1}) \ar[r]^-{f_i} &
  \unit_\mathfrak p \koszul (f_1,\ldots,f_{i-1}) \ar[r] &
    \unit_\mathfrak p \koszul (f_1,\ldots,f_{i}) \ar[r] &
}
\]
 in order to extend $x_\alpha^{i-1}$ to a nonzero map 
 $x_\alpha^i\colon \Sigma^{n_\alpha} \unit_\mathfrak p \koszul (f_1,\ldots,f_i) \to X\otimes \residue{p} $ hitting the same element in cohomology.
 In particular we obtain the announced extension~$x_\alpha^n\colon \Sigma^{n_\alpha} K(\mathfrak p)\to X\otimes K(\mathfrak p)$. 
 As a nonzero map on a one-dimensional $k(\mathfrak p)$-vector space (Lemma~\ref{lemma:koszul_quotient}), the induced map $H^*(x^n_\alpha)$ must be injective. Hence, collectively, the maps $\{x^n_\alpha\}_\alpha$ yield an isomorphism as required.
\end{proof}

\begin{prop} \label{prop:minimal}
For every $\mathfrak p$, the thick subcategory $\Thick(\residue{p})$ of~$\mathcal K_\mathfrak p$ is \emph{minimal}, meaning that it contains no proper nonzero thick subcategories.
\end{prop}

\begin{proof}
Note that for every nonzero object $X$ of $\mathcal K_\mathfrak p$ we have $X\otimes \residue{p}\neq 0$. 
Indeed if $X\otimes \residue{p}= X \koszul (f_1,\ldots, f_n)= 0$ then $X_\mathfrak p= 0$ by Lemma~\ref{lemma:ttNAK}.

Let $\mathcal C$ be a thick subcategory of $ \Thick(\residue{p})$. Because $\mathcal C$ is a tensor ideal, if it contains a nonzero object~$X$ then it also contains $X \otimes \residue{p}$, which is again nonzero by the above observation.
Therefore $\mathcal C$ must contain a shifted copy of $\residue{p}$ by Proposition~\ref{prop:decomposition}, hence $\mathcal C=\Thick (\residue{p})$. This proves the claim.
\end{proof}

\subsection*{Proof of Theorem~\ref{thm:main_spectrum}}
Now we show how to deduce our main result from the minimality of the thick subcategories $\Thick(\residue{p})$.
By Lemma~\ref{lemma:bij_homeo} it will suffice to show that the map $\rho_\mathcal K\colon \Spc \mathcal K\to \Spec R$ is bijective.
Since $R$ is graded noetherian, $\rho_\mathcal K$ is surjective by \cite[Theorem~7.3]{Balmer10b}. It remains to prove it is injective.

Let $\mathfrak p\in \Spec R$ be any homogeneous prime. We must show that the fiber of the comparison map $\rho_\mathcal K\colon \Spc \mathcal K\to \Spec R$ over $\mathfrak p$ consists of a single prime tensor ideal. 
By the pullback square \eqref{eq:pullback_spectra}, every point of $\Spc \mathcal K$ lying over $\mathfrak p$ must belong to $\Spc \mathcal K_\mathfrak p$. Hence it will suffice to show that the fiber of $\rho:=\rho_{\mathcal K_{\mathfrak p}}$ over the maximal ideal $\mathfrak m:= \mathfrak p R_{\mathfrak p}$ of $R_\mathfrak p$ consists of a single point. 
In fact a stronger statement is true:
if $\mathcal P\in \Spc \mathcal K_\mathfrak p$ is such that $\rho(\mathcal P)= \mathfrak m$, then $\mathcal P= \{0\}$.
Let us prove this.  

By definition of the comparison map we have 
\[
\rho(\mathcal P) = \langle \{ f\in R_\mathfrak p  \mid f \textrm{ is homogeneous and } \unit_\mathfrak p\koszul f \not\in \mathcal P \} \rangle\,,
\]
and as $\rho(\mathcal P)\subseteq \mathfrak m$ always holds by the maximality of~$\mathfrak m$, the hypothesis $\rho(\mathcal P) = \mathfrak m$ precisely means that 
$\unit_\mathfrak p\koszul f \not\in \mathcal P$ for all homogeneous elements $f \in \mathfrak m$.
In particular $\unit_\mathfrak p\koszul f_i \not\in \mathcal P$ for the elements $f_i$ in the chosen regular sequence for~$\mathfrak m$. 
As $\mathcal P$ is a tensor prime, we deduce further that 
\begin{equation} \label{eq:notinP}
K(\mathfrak p) \simeq  \unit_\mathfrak p\koszul f_1 \otimes \ldots \otimes \unit_\mathfrak p\koszul f_n \not\in \mathcal P \,.
\end{equation}
Now let $X\in \mathcal P$ and assume that $X\neq 0$. 
Then $X\otimes K(\mathfrak p)\neq 0$ by Lemma~\ref{lemma:ttNAK}, hence 
\begin{equation} \label{eq:equalthicks}
\Thick(X \otimes K(\mathfrak p)) = \Thick(K(\mathfrak p)) 
\end{equation}
by the minimality of $\Thick(K(\mathfrak p))$, Proposition~\ref{prop:minimal}.
As $\mathcal P$ is a thick tensor ideal we also have $X\otimes K(\mathfrak p) \in \mathcal P$ and therefore $K(\mathfrak p) \in \mathcal P$ by~\eqref{eq:equalthicks}, but this contradicts~\eqref{eq:notinP}. 
Therefore $X=0$ and we conclude that $\mathcal P=\{0\}$, proving the claim.

This concludes the proof of Theorem~\ref{thm:main_spectrum}.

\subsection*{Proof of Corollary~\ref{cor:main_thick}}
In order to deduce Corollary~\ref{cor:main_thick} from the theorem, we must verify that the homeomorphism $\rho_\mathcal K$ identifies $\SuppR H^*X \subseteq \Spec R$, the ring-theoretic support of an object $X\in \mathcal K$, with $\supp X:= \{\mathcal P \in \Spc \mathcal K \mid X\not\in \mathcal P\}$, the universal support datum of~$X$:

\begin{lemma} \label{lemma:supps_K}
We have $\SuppR H^* X = \rho_\mathcal K(\supp X)$ for all $X\in \mathcal K$.
\end{lemma}

\begin{proof}
Let $\mathfrak p=\rho_\mathcal K(\mathcal P)$. 
It follows from \eqref{eq:pullback_spectra} that $X\in \mathcal P$ iff $X_\mathfrak p\in \mathcal P_\mathfrak p$, where $\mathcal P_\mathfrak p$ denotes $\mathcal P$ seen as an element of $\Spc \mathcal K_\mathfrak p$.
We have just proved that $\rho_{\mathcal K_\mathfrak p}\colon \Spc \mathcal K_\mathfrak p \stackrel{\sim}{\to}\Spec R_\mathfrak p$ is a bijection sending $\{0\}$ to~$\mathfrak p R_\mathfrak p$, so we must have $\mathcal P_\mathfrak p=\{0\}$.
Therefore $\mathfrak p\in \SuppR H^*X$ 
$\Leftrightarrow$ $H^*X_\mathfrak p\neq0$ 
$\Leftrightarrow$ $X_\mathfrak p\neq 0$
$\Leftrightarrow$ $X_\mathfrak p\not\in\mathcal P_\mathfrak p$
$\Leftrightarrow$ $\mathcal P\in \supp X$.
\end{proof}

Now it suffices to appeal to the abstract classification theorem \cite[Theorem~4.10]{Balmer05a}. Indeed, since $R$ is noetherian, the space $\Spec R$ is noetherian and therefore its specialization closed subsets and its Thomason subsets coincide (cf.\ \cite[Remark~4.11]{Balmer05a}). Moreover, since $\mathcal K$ is generated by its tensor unit, all its objects are dualizable (because dualizable objects form a thick subcategory and $\unit$ is dualizable) and therefore all its thick tensor ideals are radical (see \cite[Proposition~2.4]{Balmer07}).  Hence by Theorem~\ref{thm:main_spectrum} and Lemma~\ref{lemma:supps_K} the classification of \cite[Theorem~4.10]{Balmer05a} immediately translates into  the classification described in Corollary~\ref{cor:main_thick}, as wished.

\section{Localizing subcategories}
\label{sec:localizing}

Assume from now on that $\mathcal T$ is a compactly generated tensor triangulated category such that its subcategory $\mathcal K:=\mathcal T^c$ of compact objects satisfies hypotheses \eqref{it:affine} and~\eqref{it:regular} from the statement of Theorem~\ref{thm:main_spectrum}.
Thus in particular $\mathcal T$ is generated as a localizing subcategory by the tensor unit: $\Loc(\unit)=\mathcal T$. If follows that every localizing subcategory of $\mathcal T$ is automatically a tensor ideal.

Since $\mathcal T$ is compactly generated, the (Verdier) $\mathfrak p$-localization functor
$q_{\mathfrak p}\colon \mathcal K\to \mathcal K_\mathfrak p$  we used so far can be extended to a finite (Bousfield) localization functor 
\[
(-)_\mathfrak p\colon \mathcal T\longrightarrow \mathcal T \,.
\]
We briefly recall its properties, referring to \cite[\S2]{BensonIyengarKrause11} or \cite[\S2]{DellAmbrogio10} for all proofs. 
Let $\mathcal L=\Loc( \{\cone(f) \mid f \in R\smallsetminus \mathfrak p \textrm{ homogeneous} \})$. 
Then the Verdier quotient  $Q\colon \mathcal T\to \mathcal T/\mathcal L=:\mathcal T_\mathfrak p$ has a fully faithful right adjoint, $I\colon \mathcal T_\mathfrak p\hookrightarrow \mathcal T$, and the functor $(-)_\mathfrak p$ can be defined to be the composite $(-)_\mathfrak p:=I\circ Q$.
As $\mathcal L$ is generated by a tensor ideal of dualizable objects, we have $X_\mathfrak p \cong X \otimes \unit_\mathfrak p$ for all $X\in \mathcal T$. 
Moreover, the unit $X\to X_\mathfrak p$ of the $(Q,I)$-adjunction induces a natural map 
 $
 \Hom^*_\mathcal T(Y,X)_\mathfrak p \to  \Hom^*_\mathcal T(Y,X_\mathfrak p)
 $
which is an isomorphism whenever $Y\in \mathcal K$ (see \cite[Proposition~2.3]{BensonIyengarKrause11} or \cite[Theorem 2.33\,(h)]{DellAmbrogio10}). In particular we have the identification
\[
(H^*X)_\mathfrak p\stackrel{\sim}{\to} H^*(X_\mathfrak p)
\]
for all~$X\in \mathcal T$.
It follows also that the restriction of $Q$ to compact objects $X,Y\in \mathcal K$ agrees with~$q_\mathfrak p$, so that we may identify $\mathcal K_\mathfrak p$ with the full subcategory $I(\mathcal K_\mathfrak p)$ of~$\mathcal T$ (and thereby eliminate the slight ambiguity of the notation ``$X_\mathfrak p$''). 

Recall the residue field objects $K(\mathfrak p)$ defined in the previous section:
\[
K(\mathfrak p) := \unit_\mathfrak p \koszul (f_1,\ldots, f_n) \simeq \unit_\mathfrak p\koszul f_1 \otimes \ldots \otimes \unit_\mathfrak p\koszul f_n \quad \in \quad \mathcal T
\]
(as before, $f_1,\ldots,f_n$ denotes the chosen regular sequence of non-zero-divisors generating the prime~$\mathfrak p$).

The main point of this section is that the crucial minimality result of Proposition~\ref{prop:minimal} can be extended to localizing subcategories of~$\mathcal T$, as we verify next.

\begin{lemma} \label{lemma:zero_action_T}
For every object $X\in \mathcal T$ and every~$i\in \{1,\ldots,n\}$, each element $f$ of  $(f_1,\ldots,f_i)\subset R$ acts as zero on $X_\mathfrak p \koszul (f_1,\ldots,f_i)$, i.e., $f\cdot {X_\mathfrak p \koszul (f_1,\ldots,f_i)}=0$.
In particular, the $R$-module $H^*(X \otimes \residue{p})$ is a graded $k(\mathfrak p)$-vector space.
\end{lemma}

\begin{proof}
Exactly the same proof as for Lemma~\ref{lemma:zero_action} and Corollary~\ref{cor:vector_space}.
(Use that $X \otimes K(\mathfrak p)=X_\mathfrak p \otimes K(\mathfrak p)$ to work inside the big $\mathfrak p$-local category $\mathcal T_\mathfrak p$.)
\end{proof}

\begin{prop} \label{prop:decomposition_T}
For all $\mathfrak p\in \Spec R$ and $X\in \mathcal T$, the tensor product $X\otimes K(\mathfrak p)$ decomposes into a coproduct of shifted copies of the residue field object: 
\[\coprod_{\alpha} \Sigma^{n_\alpha} \residue{p} \stackrel{\sim}{\longrightarrow} X\otimes \residue{p}\,.\]
\end{prop}

\begin{proof}
Exactly the same as for Proposition~\ref{prop:decomposition}, using Lemma~\ref{lemma:zero_action_T}.
\end{proof}

\begin{prop} \label{prop:minimal_T}
For every $\mathfrak p$, the localizing subcategory $\Loc(\residue{p})$ of~$\mathcal T$ is minimal, meaning that it contains no proper nonzero localizing subcategories.
\end{prop}

\begin{proof}
This follows from Proposition~\ref{prop:decomposition_T} precisely as in the proof of Proposition~\ref{prop:minimal}, except that we cannot use Lemma~\ref{lemma:ttNAK} to show that $X\otimes \residue{p}\neq 0$ for every nonzero object $X\in \Loc(\residue{p})$.
Instead, we may use the following argument. 

First note that $X \otimes \residue{q} =0$ for all  $\mathfrak q \in \Spec R \smallsetminus \{\mathfrak p\}$.
Indeed, this property holds for $X=\residue{p}$ by Lemma~\ref{lemma:zero_action_T} (because if $\mathfrak p\neq \mathfrak q$ then some homogeneous element of $R$ must act on $\residue{p}\otimes \residue{q}$ both as zero and invertibly) and is stable under taking coproducts and mapping cones (as the latter are preserved by $-\otimes \residue{p}$); hence it must hold for all objects of $\Loc(\residue{p})$, as wished.
Now combine this with Proposition~\ref{prop:detection} below.
\end{proof}

\begin{lemma} \label{lemma:min_supp}
Let $M$ be any nonzero module, possibly infinitely generated, over a noetherian $\mathbb Z$-graded commutative ring~$S$. 
Then there exists a minimal prime in $\mathrm{Supp}_S M : = \{ \mathfrak p \in \Spec S \mid M_\mathfrak p \neq 0 \}$, the big Zariski support of~$M$.
\end{lemma}

\begin{proof}
If $M\neq 0$ then $M_\mathfrak p\neq 0$ for some prime~$\mathfrak p$, so the support is not empty. 
Moreover, it suffices to prove the claim for the nonzero module $M_\mathfrak p$ over~$S_\mathfrak p$, because a minimal prime of $\mathrm{Supp}_{S_\mathfrak p} M_\mathfrak p$ yields a minimal prime in $\mathrm{Supp}_S M$; hence we may assume that $S$ is local.
 By Zorn's lemma it suffices to show that in $\mathrm{Supp}_S M$ every chain of primes admits a minimum. Indeed, each such chain must stabilize, because a local commutative noetherian ring has finite Krull dimension. In the ungraded case, the latter is a well-known corollary of Krull's principal ideal theorem. A proof of the analogous result for graded rings can be found in~\cite[Theorem~1.5.8]{BurnsHerzog93} or~\cite[Theorem~3.5]{ParkPark11}.
\end{proof}

\begin{prop} \label{prop:detection}
If an object $X\in \mathcal T$ is such that $X\otimes \residue{p}= 0$ for all $\mathfrak p\in \Spec R$ then $X= 0$. 
\end{prop}

\begin{proof}
We prove the contrapositive. Assume that $X\neq 0$. Then $H^*X\neq 0$, hence  for some~$\mathfrak{p}\in \Spec R$ we must have $H^*(X_\mathfrak p)=(H^*X)_\mathfrak p\neq0$ and therefore $X_\mathfrak p\neq 0$. 
By Lemma~\ref{lemma:min_supp}, we may choose a prime $\mathfrak p$ which is minimal among the primes with this property. 
Thus the big support of the $R$-module $H^*X_\mathfrak p$ consists precisely of the prime~$\mathfrak p$.
We are going to recursively show that 
$X_i:=X_\mathfrak p\koszul (f_1,\ldots,f_i)$  
satisfies $\Supp_R H^* X_i = \{\mathfrak p\}$ for all $i\in\{1,\ldots,n\}$.
Thus in particular $X\otimes \residue{p}  = X_n \neq 0$, which proves the proposition. 
We already know that $\Supp_R H^* X_0 = \{\mathfrak p\}$ for $X_0:=X_\mathfrak p$, and suppose we have  shown that 
$\Supp_R H^* X_{i-1} = \{\mathfrak p\}$.
The exact triangle 
\[
\xymatrix{
\Sigma^{-|f_i|}X_{i-1} \ar[r]^-{f_i} & X_{i-1} \ar[r] & X_i \ar[r] & \Sigma^{-|f_i|+1} X_{i-1} 
}
\]
implies that $\Supp_R H^* X_i \subseteq \{\mathfrak p\}$. Hence $X_i\neq 0$ is equivalent to $\Supp_R H^* X_i = \{\mathfrak p\}$. 
By the triangle again, if $X_i= 0$ were the case $f_i$ would act invertibly on $X_{i-1}$ and thus on $H^*X_{i-1}$.
This implies $H^* X_{i-1} = (H^* X_{i-1})[f_i^{-1}] $, and since $f_i \in \mathfrak p$ we would conclude that $\mathfrak p\not\in \Supp_R H^* X_{i-1}$, in contradiction with the induction hypothesis. Therefore $X_i\neq 0$, as claimed.
\end{proof}

\subsection*{Proof of Theorem~\ref{thm:main_localizing}}
The result now follows easily from the machinery developed by Benson, Iyengar and Krause in \cite{BensonIyengarKrause08} and \cite{BensonIyengarKrause11}. 
Indeed, by \cite[Theorem~4.2]{BensonIyengarKrause11} in order to obtain the claimed classification of localizing subcategories it suffices to verify that \emph{the action of $R$ stratifies $\mathcal T$}. By definition, this means that the following two axioms are satisfied: 
\begin{itemize}
\item \emph{The local-global principle:} For every object $X\in \mathcal T$ we have the equality
\[
\Loc(X) = \Loc(\{\Coloc{\mathfrak p}X \mid \mathfrak p \in \Spec R \})
\]
of localizing subcategories of~$\mathcal T$.
\item \emph{Minimality:} For every $\mathfrak p\in \Spec R$ the localizing subcategory $\Coloc{\mathfrak p}\mathcal T$ of $\mathcal T$ is minimal or zero.
\end{itemize}
The functors $\Coloc{\mathfrak p}\colon \mathcal T\to \mathcal T$ are introduced in \cite{BensonIyengarKrause08}, but we don't need to know how they are defined.
In our context, i.e.\ where $\mathcal T$ is a tensor category and the action of $R$ is the canonical one of the central ring, the local-global principle always holds by \cite[Theorem~7.2]{BensonIyengarKrause11} (see also \cite[Theorem~6.8]{Stevenson13}).
Moreover $\Coloc{\mathfrak p} X = X\otimes \Coloc{\mathfrak p} \unit$ for all $X\in \mathcal T$, which implies that $\Coloc{\mathfrak p}\mathcal T = \Loc(\Coloc{\mathfrak p}\unit)$ since $\mathcal T$ is generated by~$\unit$.
Therefore the remaining minimality condition follows from Proposition~\ref{prop:minimal_T}, because $\Loc(\residue{p}) = \Loc(\Coloc{\mathfrak p}\unit)$ by \cite[Lemma 3.8~(2)]{BensonIyengarKrause11} (indeed, by construction $\residue{p}$ is a particular instance of the objects collectively denoted by $\unit(\mathfrak p)$ in \emph{loc.\,cit.}).

This establishes the first bijection in Theorem~\ref{thm:main_localizing}.

The claimed identification of the Benson-Iyengar-Krause support,
 $\supp_R X = \{ \mathfrak p\in \Spec R\mid X \otimes \Coloc{\mathfrak p} \unit \neq 0\}$, 
 with the set $\{\mathfrak p\in \Spec R\mid X\otimes \residue{p}\neq0\}$ is an easy consequence of the equality $\Loc(\residue{p}) = \Loc(\Coloc{\mathfrak p}\unit)$ mentioned above.
 
 It remains to verify the moreover part of Theorem~\ref{thm:main_localizing}.
Let us begin by noting that, if $X\in \mathcal K$ is a compact object, we have
\begin{equation} \label{eq:supps_cpt}
\suppR X= \suppR H^*X =\SuppR H^*X
\end{equation}
by \cite[Theorem 5.5\,(1)]{BensonIyengarKrause08} and Lemma~\ref{lemma:supports}.

Now let $\mathcal L\subseteq \mathcal T$ be such that $\mathcal L= \Loc(\mathcal L \cap \mathcal K)$. 
Then 
\[
\bigcup_{X\in \mathcal L} \suppR X = \bigcup_{X\in \mathcal L\cap \mathcal K} \suppR X = \bigcup_{X\in \mathcal L\cap \mathcal K} \SuppR H^*X
\] 
by \eqref{eq:supps_cpt}, and the latter is a specialization closed subset of the spectrum. 
Conversely, if $S \subseteq \Spec R$ is specialization closed the corresponding localizing subcategory $\{X\in \mathcal T\mid \suppR X\subseteq S\}$ is generated by compact objects by \cite[Theorem~6.4]{BensonIyengarKrause08}, hence $\mathcal L=\Loc(\mathcal L\cap \mathcal K)$.
This concludes the proof of the theorem.

It is well-known that the assignments $\mathcal C\mapsto \Loc(\mathcal C)$ and $\mathcal L \mapsto \mathcal L\cap \mathcal K$ are mutually inverse bijections between thick subcategories $\mathcal C\subseteq \mathcal K$ and localizing subcategories $\mathcal L\subseteq \mathcal T$ which are generated by compact objects of~$\mathcal T$ (see~\cite{Neeman92b}). Together with \eqref{eq:supps_cpt}, this shows how to deduce the classification of thick subcategories of Corollary~\ref{cor:main_thick} from Theorem~\ref{thm:main_localizing}. 

Finally, there are several ways to derive the telescope conjecture of Corollary~\ref{cor:telescope} from the previous results. For instance, we may proceed as in \cite[\S6.2]{BensonIyengarKrause11}.

\begin{remark}
Using the theory of coherent functors, Benson, Iyengar and Krause have recently developed in  \cite{BensonIyengarKrause15} an analogue of their stratification theory of compactly generated categories that can be applied to general essentially small triangulated categories. Their theory, and more specifically \cite[Theorem~7.4]{BensonIyengarKrause15}, provides an alternative way to derive Theorem~\ref{thm:main_spectrum} from Proposition~\ref{prop:minimal}.
\end{remark}

\subsection*{The case of commutative dg~algebras}

We still owe our reader a proof of Theorem~\ref{thm:cdga}. Let $A$ be a commutative dg~algebra and let $D(A)$ be the derived category of (left, say) dg-$A$-modules. The following elementary fact was pointed out to us by the referee.

\begin{lemma} \label{lemma:cdga}
Every $f\in H^*A$ acts as zero on its own mapping cone~$C(f)$.
\end{lemma}

\begin{proof}
A (homogeneous) element $f\in H^*A$ of degree $|f|=-n$ is (represented by) a morphism $f\colon \Sigma^n A \to A$ of left dg-$A$-modules. Let us write $\s a$  ($a\in A$) for a generic element of degree $|a|-1$ in the suspension $\s A := \Sigma A$; here we use the Koszul sign convention and treat $\s$ as a symbol of degree~$-1$. The cone $C(f)$ has elements $(a,\s^{n+1} b)$ (for $a,b\in A$). Then $f$ acts on $C(f)$ by a morphism $\s^n C(f)\to C(f)$ which, under the isomorphism $\s^n C(f) \cong C(\s^n f)$, is written as follows: 
\[ g\colon C(\s^n f)\to C(f) \quad \quad g(\s^na , \s^{2n+1} b) = (f(\s^na), \s^{n+1} f(\s^nb))\] 
(recall that the suspension $\s h\colon \s B\to \s C$ of a morphism $h\colon B\to C$ is given by $(\s h)(\s b)= \s (h(b))$).
With these notations, the map $H\colon C(\s^nf)\to C(f)$ defined by $H(\s^na, \s^{2n+1}b):=(0,\s^{n+1}a)$ is easily seen to satisfy $H(t x)= (-1)^{|t|} t H(x)$ (for $t\in A, x\in C(\s^nf)$) and $dH + Hd = -g$; in other words, $H$ is a homotopy $g\sim 0$ defined over~$A$.
\end{proof}

As noted in Remark~\ref{rem:hyp}, Lemma~\ref{lemma:zero_action} was the only place in all of our arguments where we made use of the hypothesis that $R$ is concentrated in even degrees and that in the regular sequences we may choose the elements to be non-zero-divisors. 
But if we consider the example $\mathcal K:=D(A)^c$, $\mathcal T:=D(A)$ and $R:=H^*A$, we see immediately that the conclusion of the lemma also follows from the above Lemma~\ref{lemma:cdga}. Hence in this case we can get rid of the extra hypotheses, while the rest of our arguments go through unchanged. This proves Theorem~\ref{thm:cdga}.

Indeed, in general in Theorem~\ref{thm:main_spectrum}~(b) we could similarly renounce the evenness of $R$ if we substitute the requirement that all elements $f_i$ of the regular sequences be non-zero-divisors with the requirement that $f_i\cdot \unit \koszul f_i=0$.

\subsection*{Acknowledgements}

We are very grateful to Paul Balmer and an anonymous referee for their useful comments, and to the referee for suggesting Theorem~\ref{thm:cdga}. We also thank Joseph Chuang for the reference~\cite{Mathew14}. 

\bibliographystyle{alpha}%
\bibliography{ttg_articles}

\end{document}